\documentclass{amsart}
\usepackage{amsmath}
\usepackage{framed,comment,enumerate}
\usepackage[pdftex]{graphicx}
\usepackage{epstopdf}
\usepackage{xcolor, framed}
\usepackage{color}
\usepackage{hyperref}
\usepackage{mathtools}
\usepackage{todonotes}

\def\vem{\vspace{1em}}
\def\hem{\hspace{0.5em}}

\def\R {\mathbb{R}}
\def\PunS{\R^d\backslash\{0\}}

\def\eps{\varepsilon}
\def\Sd{\mathcal{S}_d}

\def\LSet{\Lambda(u)}
\def\PosS{\{u>0\}}

\def\GC{\mathcal{G}_c}
\def\PC{\mathcal{P}_c}
\def\Q{\mathcal{Q}}

\def\Sir{\Psi_r}
\def\SiR{\Psi_{\bar{r}}}

\newtheorem{thm}{Theorem}[section]
\newtheorem{prop}[thm]{Proposition}

\newtheorem{cor}[thm]{Corollary}
\newtheorem{lem}[thm]{Lemma}
\theoremstyle{definition}

\newtheorem{rem}[thm]{Remark}
\numberwithin{equation}{section}

\title[Compact Contact Sets]{Solutions to the nonlinear obstacle problem with compact contact sets} 

\author{Simon Eberle}
\address{Basque Center for Applied Mathematics, Bilbao, Spain}
\email{seberle@bcamath.org}

\author{Hui Yu}
\address{Department of Mathematics,	National University of Singapore, Singapore}
\email{ huiyu@nus.edu.sg}

\begin{document}

\begin{abstract}
For the obstacle problem with a nonlinear operator, we characterize the space of global solutions with compact contact sets. This is achieved by constructing  a bijection onto a class of quadratic polynomials describing the asymptotic behavior of  solutions. 
\end{abstract}

\maketitle
\section{Introduction}
Solutions to elliptic problems in the entire space often exhibit rigidity properties. This is exemplified by the Liouville theorem, which says that global harmonic functions with polynomial growths are indeed polynomials. In particular, such solutions form a finite dimensional space. 

In this article, we explore the rigidity of the nonlinear obstacle problem. 

For an integer $d\ge3$,  let $\mathcal{S}_d$ denote the space of $d$-by-$d$ symmetric matrices, and suppose that $F:\mathcal{S}_d\to\R$ is an elliptic operator that is $C^1$ and convex. We study \textit{global solutions} to the \textit{nonlinear obstacle problem}, namely, 
\begin{equation}
\label{EqObP}
\begin{cases}
F(D^2u)\le 1 &\text{ in }\R^d,\\
u\ge 0 &\text{ in }\R^d,\\
F(D^2u)=1 &\text{ in }\PosS.
\end{cases}
\end{equation} 
Given a solution $u$, we refer to its zero set as the \textit{contact set}, to be denoted as 
\begin{equation}
\label{EqConS}
\Lambda(u):=\{u=0\}.
\end{equation} 

For the theory of fully nonlinear elliptic operators, the reader may consult Caffarelli-Cabr\'e \cite{CC}. For an introduction to the obstacle problem, the reader may refer to Caffarelli \cite{C} and Petrosyan-Shahgholian-Uraltseva \cite{PSU}. The nonlinear obstacle problem was first studied by Lee \cite{L}. See also Savin-Yu \cite{SY1, SY2} for recent developments.

\vem 

When $F$ is the trace\footnote{That is, when the operator is the Laplacian.},  the system \eqref{EqObP} is known as the \textit{classical obstacle problem}, that is,
\begin{equation}
\label{EqIntroClassicalObP}
\Delta u\le 1, \hem u\ge 0 \text{ in }\R^d, \hem \text{ and }\Delta u=1 \text{ in }\PosS.
\end{equation} 
Global solutions for this problem have been completely classified. This was first done for solutions with compact contact sets by Dive \cite{Di}, DiBenedetto-Friedman \cite{DF} and Friedman-Sakai \cite{FSa}. See Eberle-Weiss \cite{EW} for a short proof. Without the compactness of $\LSet$, the classification was established by Sakai \cite{Sa} in two dimensions. For solutions in general dimensions with non-compact contact sets, the classification was recently obtained by Eberle-Shahgholian-Weiss \cite{ESW} and Eberle-Figalli-Weiss \cite{EFW}, resolving a conjecture posed by Shahgholian \cite{Sh} and Karp-Margulis \cite{KM}.

Thanks to these works, we know that for the classical obstacle problem, a global solution can be expanded as
\begin{equation}
\label{EqIntroExpansion}
u=p+v,
\end{equation}
where $p$ is a quadratic polynomial and $v$ is a (generalized) Newtonian potential. Geometrically, the family of contact sets consists of limits of ellipsoids.

A closely related problem is the \textit{thin obstacle problem}, namely,
\begin{equation}
\label{EqIntroTObP}
\Delta u\le 0 \text{ in }\R^d,\hem
u\ge 0 \text{ on }\{x_d=0\}, \hem \text{ and }
\Delta u=0 \text{ in }\{x_d\neq0\}\cup\PosS.
\end{equation} 
For this problem, a similar expansion as in \eqref{EqIntroExpansion} has been established by Eberle, Ros-Oton and Weiss \cite{ERW} for global solutions with compact contact sets. For sub-quadratic solutions, compact contact sets have been fully characterized by Eberle-Yu \cite{EY}.

\vem

In these works on \eqref{EqIntroClassicalObP} and \eqref{EqIntroTObP}, two crucial ingredients are monotonicity formulae and the explicit relation between a set and its Newtonian potential. Due to its nonlinearity, these tools no longer apply to our problem \eqref{EqObP}. 

As a first step towards a complete classification, in this work we restrict our attention to solutions with compact contact sets. To simplify our exposition, we denote the class of such solutions by
\begin{equation}
\label{EqGC}
\GC:=\{u:\hem u \text{ solves }\eqref{EqObP} \text{ in }\R^d, \hem \text{ and }\Lambda(u)\neq \emptyset \text{ is compact}\}.
\end{equation} 

With the absence of invariance properties of the Laplacian, a geometric characterization of contact sets seems out of reach. Nevertheless, for solutions in $\GC$, we are able to establish an expansion like  \eqref{EqIntroExpansion}. In particular,  the space $\GC$ is finite dimensional, a form of rigidity of the nonlinear obstacle problem  \eqref{EqObP}.

To be precise, we introduce the class of quadratic polynomials, $\PC$, that appear in the expansion. 
For a quadratic polynomial $Q$, we say that 
\begin{equation}
\label{EqPC}
Q\in\PC
\end{equation} 
if it satisfies the following conditions:
\begin{enumerate}
\item{The polynomial is of the form $Q=\frac12 (x-\overline{x})\cdot A(x-\overline{x})+a$ for some $A\in\Sd$, $\overline{x}\in\R^d$ and $a\le0;$}
\item{The coefficient matrix $A$ satisfies $F(A)=1$; and }
\item{Eigenvalues of $A$ are strictly positive.}
\end{enumerate}

With this, our main result reads
\begin{thm}
\label{ThmMainThm}
For $d\ge3$, suppose that $F$ is a uniformly elliptic operator that is $C^1$ and convex, then 
there is a bijection $\Phi:\GC\to\PC$ such that 
$$
|u-\Phi(u)|(x)= O(|x|^{2-d})\text{ as }|x|\to\infty.
$$
\end{thm} 

\begin{rem}
If we let $Q:=\Phi(u)$ and $v:=u-Q$, then $v$ 
 solves a nonlinear obstacle problem with $-Q$ as the obstacle, that is, 
$$
G(D^2v)\le 0, \hem v\ge-Q \text{ in }\R^d, \hem\text{ and }G(D^2v)=0 \text{ in }\{v>-Q\}.
$$
Here $G:\mathcal{S}_d\to\R$ is the operator defined as
$$
G(M):=F(M+D^2Q)-1 \text{ for }M\in\mathcal{S}_d.
$$
\end{rem}

\vem
This short note is structured as follows: In Section 2, we collect some preliminaries. In Section 3, we construct the map $\Phi$ in Theorem \ref{ThmMainThm} and establish its bijectivity.

\section{Preliminaries}
In this section, we collect  some properties of solutions to the nonlinear obstacle problem \eqref{EqObP}. 

For the operator $F$, we assume the following:
\begin{enumerate}
\item{The operator $F$ is uniformly elliptic, that is, there is a constant $\Lambda\in[1,+\infty)$ such that 
$$
\frac{1}{\Lambda}\|P\|\le F(M+P)-F(M)\le\Lambda\|P\|
$$
for all $M,P\in\Sd$ and $P\ge0;$}
\item{The operator $F$ is convex; and }
\item{The operator $F$ is $C^1$.}
\end{enumerate}

\vem

We begin with the comparison principle for \eqref{EqObP}. Solutions, as well as subsolutions and supersolutions,  are understood in the viscosity sense \cite{CC}. In particular, they are continuous. 
\begin{prop}
\label{PropComparison}
Suppose that  $u$ solves the nonlinear obstacle problem \eqref{EqObP} in $B_1\subset\R^d$. 

1) If $v$ solves \eqref{EqObP} in $B_1$ with 
$$
v\ge u \text{ on }\partial B_1,
$$
then
$$
v\ge u \text{ in } B_1.
$$

2) If $w$ satisfies 
$$
F(D^2w)\ge 1\text{ in }B_1, \hem \text{ and }w\le u \text{ on }\partial B_1,
$$
then 
$$
w\le u \text{ in }B_1.
$$
\end{prop} 

\begin{proof}
The second statement is a consequence of the  comparison principle as we have 
$$
F(D^2u)\le 1\le F(D^2w) \text{ in }B_1.
$$
It remains to prove the first statement. 

By continuity, we find $t\in\R$ such that 
$$
v+t\ge u \text{ in }\overline{B_1}, \text{ and }v(x_0)+t=u(x_0) \text{ at some }x_0\in\overline{B_1}.
$$
It suffices to show that 
$
t\le0.
$

Suppose $t>0$, then  with the comparison on $\partial B_1$, we have $x_0\in B_1$. With $v\ge0$, we have $x_0\in\PosS.$ This contradicts the comparison principle in the domain $\PosS$, where $F(D^2u)=1\ge F(D^2v)$.
\end{proof} 

As a corollary, we have
\begin{cor}
\label{CorComparisonForGlobal}
Suppose that $u$ and $v$ are two global solutions to the nonlinear obstacle problem \eqref{EqObP} with 
$$
\liminf_{|x|\to\infty} (u-v)(x)\ge0,
$$
then 
$$
u\ge v \text{ in }\R^d.
$$
\end{cor}

\begin{proof}
Take $x_0\in\R^d$ and $\eps>0$, then for large $R$, we have $x_0\in B_R$ and 
$$
u\ge v-\eps \text{ on }\partial B_R.
$$
Proposition \ref{PropComparison} gives $u(x_0)\ge v(x_0)-\eps$. 

This being true for all $\eps>0$, we have $u(x_0)\ge v(x_0).$
\end{proof}

The following provides compactness for the problem. For the nonlinear problem, it was first established by Lee \cite{L}.
\begin{prop}
\label{PropC11}
Suppose that $u$ solves the  \eqref{EqObP} in $B_1\subset\R^d$ with
$
u(0)=0.
$
Then
$$
\sup_{B_r}u\le Cr^2 \text{ for all  } r<1/2,
$$
and
$$
\|u\|_{C^{1,1}(B_{1/2})}\le C
$$
for a constant $C$ depending only on the dimension $d$ and the ellipticity constant $\Lambda$.
\end{prop} 

The following gives the stability of the problem. See, for instance, Proposition 3.17 in Petrosyan-Shahgholian-Uraltseva \cite{PSU}.

\begin{prop}
\label{PropStab}
Suppose that $\{u_n\}$ is a sequence of solutions to \eqref{EqObP} in $B_1\subset\R^d$ and 
$$
u_n\to u_\infty \text{ locally uniformly in }B_1.
$$
Then the convergence is locally uniform in $W^{2,p}$ for any $p\in [1,+\infty)$, and the limit $u_\infty$ solves \eqref{EqObP} in $B_1$.
\end{prop} 

With these two propositions, we classify blow-downs of  solutions to our problem. Recall  the class of solutions $\GC$ from \eqref{EqPC}. 

\begin{lem}
\label{LemBlowDown}
For  $u\in\GC$ and $R>0$, define the rescaled solution
$$
u_R(x):=u(Rx)/R^2.
$$
Then there is a subsequence $R_n\to+\infty$ such that 
$$
u_{R_n}\to p=\frac12 x\cdot Ax \text{ locally uniformly in }\R^d,
$$
where $A\ge0$ and $F(A)=1.$
\end{lem} 

\begin{proof}
With the contact set $\LSet\neq\emptyset$,  Proposition \ref{PropC11} implies that the family $\{u_R\}$ is locally uniformly bounded in $C^{1,1}$. This allows us to extract a subsequence $\{u_{R_n}\}$ that locally uniformly converges to $u_\infty$, a solution to \eqref{EqObP} by Proposition \ref{PropStab}. With Proposition \ref{PropC11}, we have
\begin{equation}
\label{QuadraticGrowth}
\sup_{B_r}u_\infty\le Cr^2 \text{ for all }r>0.
\end{equation}

With $W^{2,p}_{loc}$-convergence in Proposition \ref{PropStab}, we have
$$
F(D^2u_\infty)=\lim_{n\to\infty} F(D^2u_{R_n})=\lim_{n\to\infty}\chi_{\{u_{R_n}>0\}}=1-\lim_{n\to\infty}\chi_{\{u_{R_n}=0\}}=1.
$$
For the last equality, we used 
$$
\{u_{R_n}=0\}=\frac{1}{R_n}\LSet,
$$
and that $\LSet$ is compact as $u\in\GC$.

Consequently, the limit $u_\infty$ solves $F(D^2u_\infty)=1$ in $\R^d$. With \eqref{QuadraticGrowth}, the Evans-Krylov estimate (Theorem 6.6 in \cite{CC}) gives, for all $r>0$,
$$
|D^2u_{\infty}(x)-D^2u_\infty(y)|\le C/r^\alpha\cdot |x-y|^\alpha \text{ for }x,y\in B_r.
$$
Here $\alpha\in(0,1)$ is a dimensional constant. 

Sending $r\to\infty$, we have that $D^2u_\infty$ is constant, and that $u_\infty$ is a quadratic polynomial. Estimate \eqref{QuadraticGrowth} for small $r$ implies that the linear and constant order terms vanish in this polynomial. 
\end{proof} 

For simplicity, we denote the family of blow-downs as $\Q$, that is,
\begin{equation}
\label{EqQ}
\Q:=\{p=\frac12 x\cdot Ax:A\ge0, \hem F(A)=1\}.
\end{equation} 

With the compactness assumption on the contact set $\LSet$, we often encounter solutions to elliptic equations in exterior domains. The following Liouville-type result by Li-Li-Yuan \cite{LLY} provides a powerful tool to address such solutions.
\begin{thm}[Theorem 2.1 in \cite{LLY}]
\label{ThmLLY}
Let $G$ be a uniformly elliptic operator that is $C^1$ and convex. For $d\ge3,$ suppose that $v$ satisfies
$$
G(D^2v)=0 \text{ in }\R^d\backslash B_1, \hem \text{ and }\sup_{\R^d\backslash B_1}|D^2v|<\infty.
$$
Then there exists a unique quadratic polynomial 
$$
Q(x)=\frac12 x\cdot Ax+b\cdot x+c
$$
such that 
$$
|v-Q|(x)=O(|x|^{2-d}) \text{ as }|x|\to\infty.
$$
\end{thm}

\section{The map and its bijectivity}
In this section, we construct the map $\Phi$ in Theorem \ref{ThmMainThm} and show that it is a bijection between the spaces $\GC$ and $\PC$ defined in \eqref{EqGC} and \eqref{EqPC} respectively.

\subsection{Construction of the map}
The following proposition defines the map $\Phi$.
\begin{prop}
\label{ThisProposition}
For $d\ge 3$, 
given $u\in\GC$, there is a unique $Q\in\PC$ such that 
$$
|u-Q|(x)=O(|x|^{2-d}) \text{ as }|x|\to\infty.
$$
\end{prop} 
\begin{rem}
\label{RemDefOfPhi}
Once Proposition \ref{ThisProposition} is established, the map $\Phi$ will be defined as 
$$
\Phi(u):=Q.
$$
\end{rem}

\begin{proof}
Let $p=\frac12 x\cdot Ax\in\Q$ be a blow-down of $u$ given by Lemma \ref{LemBlowDown}, where the class $\Q$ is defined in \eqref{EqQ}. Since the contact set $\LSet$ is compact, up to a rescaling, we assume 
$$
\LSet\subset B_1.
$$

If we let 
$
v:=u-p,
$
and define $G:\Sd\to\R$ by
$
G(M):=F(M+A)-F(A),
$
then $v$ satisifes
$$
G(D^2v)=0 \text{ outside }B_1.
$$
Moreover, with Proposition \ref{PropC11}, we have 
$
\sup_{\R^d}|D^2v|< \infty.
$
Consequently, Theorem \ref{ThmLLY} gives
$$
v(x)=\frac12 x\cdot Mx+b\cdot x+a+E(x),
$$
with 
$$
|E(x)|=O(|x|^{2-d}).
$$

Define the quadratic polynomial 
$$
Q:=p+\frac12 x\cdot Mx+b\cdot x+a=\frac12 x\cdot (A+M)x+b\cdot x+a.
$$
Then we have
\begin{equation}
\label{ProofOfProp3.1Expansion}
u(x)=Q(x)+E(x).
\end{equation}
It is not difficult to see such polynomial $Q$ is unique. 

In the remaining part, we show that $Q$ belongs to $\PC$ from \eqref{EqPC}.

\vem 

\textit{Step 1: Identification of the quadratic part.}

For each $R>0$, define the rescaled solution 
$
u_R(x)=\frac{1}{R^2}u(Rx).
$
For these functions, the expansion in \eqref{ProofOfProp3.1Expansion} reads
$$
u_R(x)=\frac12 x\cdot(A+M)x+\frac{1}{R}b\cdot x+\frac{a}{R^2}+\frac{1}{R^2}E(Rx).
$$
By the definition of the blow-down $p=\frac12 x\cdot Ax$, we have $u_{R_n}\to p$ along a sequence of $R_n\to+\infty$. This implies
$
M=0.
$
As a result,
$$
u(x)=\frac12 x\cdot Ax+b\cdot x+a+E(x).
$$

\vem

\textit{Step 2: Absorption of the linear part.}

Up to a rotation, we have 
$$
u(x)=\frac12\sum a_jx_j^2+\sum b_jx_j+a+E(x),
$$
where $a_j\ge0.$

Suppose $a_1=0$, then by the  boundedness of $a+E(x)$ and the fact $u\ge0$, we have 
$b_1=0$. Similarly, for $j=1,2,\dots,d$, we must have $b_j=0$ whenever $a_j=0.$

Consequently, we have
$$
u(x)=\frac12\sum a_j(x_j-\overline{x}_j)^2+a'+E(x)
$$
for a point $\overline{x}=(\overline{x}_1,\overline{x}_2,\dots,\overline{x}_d)\in\R^d$ and $a'\in\R$.

\vem

\textit{Step 3: Sign of the constant. }

Up to now, we have
$$
u(x)=\frac12 (x-\overline{x})\cdot A(x-\overline{x})+a'+E(x).
$$

Suppose that $a'>0$. With $A\ge0$ and $F(A)=1$, we see that $q:=\frac12 (x-\overline{x})\cdot A(x-\overline{x})+\frac12 a'$ is a solution to the nonlinear obstacle problem \eqref{EqObP}. 

With $|E(x)|=O(|x|^{2-d})$, we have
$$
\liminf_{|x|\to\infty}(u-q)(x)\ge 0.
$$
Corollary \ref{CorComparisonForGlobal} implies 
$$
u\ge q>0 \text{ in }\R^d.
$$
In particular, we have $\LSet=\emptyset$, contradicting our assumption that $u\in\GC$. 

Therefore, we must have
$$
a'\le 0.
$$

\vem

\textit{Step 4: Non-degeneracy of the quadratic part.}

If $a'<0$, then with $|E(x)|=O(|x|^{2-d})$ and $u\ge0$, we have $\frac{1}{2}(x-\overline{x})\cdot A(x-\overline{x})>0$ for all large $x$. In this case, eigenvalues of $A$ are strictly positive. 

If $a'=0$, then we can apply Corollary \ref{CorComparisonForGlobal} to $u$ and $\frac12(x-\overline{x})\cdot A(x-\overline{x})$ to conclude
$$
u=\frac12(x-\overline{x})\cdot A(x-\overline{x}) \text{ in }\R^d.
$$
With the compactness of $\LSet$, again we conclude that all eigenvalues of $A$ are strictly positive. 

\vem 

In summary, we have 
$$
Q(x)=\frac12 (x-\overline{x})\cdot A(x-\overline{x})+a',
$$
where $F(A)=1$ and $a'\le 0$. Moreover, all eigenvalues of $A$ are strictly positive.  Therefore, we have $Q\in\PC.$
\end{proof} 

\subsection{Injectivity of the map} 
As a consequence of Corollary \ref{CorComparisonForGlobal}, we have the following.
\begin{prop}
For $d\ge3,$
suppose that $u$ and $v$ are global solutions to the nonlinear obstacle problem \eqref{EqObP}, and let $\Phi$ be the map defined in Remark \ref{RemDefOfPhi}.

If $\Phi(u)=\Phi(v)$, then $u=v$.
\end{prop} 
\begin{proof}
With $\Phi(u)=\Phi(v)$, we have
$$
\lim_{|x|\to\infty}(u-v)(x)=0.
$$
The conclusion follows from Corollary \ref{CorComparisonForGlobal}.
\end{proof}

\subsection{Surjectivity of the map}
Given a quadratic polynomial $$Q=\frac12(x-\overline{x})\cdot A(x-\overline{x})+a\in\PC,$$ in this subsection we construct  $u\in\GC$ such that 
$$
\Phi(u)=Q.
$$
Recall the space of solutions $\GC$ and the space of polynomials $\PC$ from \eqref{EqGC} and \eqref{EqPC} respectively.  The map $\Phi$ is defined in Remark \ref{RemDefOfPhi}.

By translation, we might take $\overline{x}=0.$

If $a=0$, then $Q\in\GC$ and $$\Phi(Q)=Q.$$ Thus it remains to consider the case 
$a<0$. Up to a rescaling, we assume  $a=-1.$

If we let $DF(\cdot)$ denote the derivative of $F$, then up to a normalization, it suffices to consider the case when 
$$
DF(A)=I.
$$

Consequently, for the surjectivity of the map $\Phi$, it suffices to establish the following:
\begin{prop}
For $d\ge3, $
given $Q=\frac12 x\cdot Ax-1\in\PC$ with $DF(A)=I$, there is $u\in\GC$ such that 
$$
|u-Q|(x)=O(|x|^{2-d}) \text{ as }|x|\to\infty.
$$
\end{prop}

\begin{proof}
\textit{Step 1: A global solution.}

For each $R>0$, let $u_R$ be the solution to the following problem 
$$
\begin{cases}
F(D^2u_R)\le 1 &\text{ in $B_R$,}\\
u_R\ge 0 &\text{ in }B_R,\\
F(D^2u_R)=1 &\text{ in }\{u_R>0\},\\
u_R=Q &\text{ on }\partial B_R.
\end{cases}
$$ 
Applying Proposition \ref{PropComparison} to $u_R$, $\frac12 x\cdot Ax$ and $Q$, we have
\begin{equation}
\label{EqTrappedInBetween}
Q\le u_R\le\frac12x\cdot Ax \text{ in }B_R.
\end{equation}
This leads to the following inclusion
\begin{equation*}
\label{EqInclusionOfContact}
0\in\Lambda(u_R)\subset\{Q\le 0\},
\end{equation*}
where $\Lambda(u_R)$ denote the contact set of $u_R$ as in \eqref{EqConS}. Since eigenvalues of $A$ are positive, the set $\{Q\le 0\}$ is compact. Up to a rescaling, we can assume
$$
\{Q\le0\}\subset B_1.
$$
This  implies, for $R>1$, 
\begin{equation}
\label{EqEqOutSide}
F(D^2u_R)=1 \text{ in }B_R\backslash B_1.
\end{equation} 

For $R'>R>0$, we have $u_{R'}\ge Q=u_{R}$ on $\partial B_R$. Proposition \ref{PropComparison} implies
$
u_{R'}\ge u_R \text{ in } B_R.
$
Together with Proposition \ref{PropC11} and Proposition \ref{PropStab}, the sequence $\{u_R\}$ satisfies
$$
u_R\to u_\infty \text{ as }R\to\infty
$$
locally uniformly in $\R^d$, where $u_\infty$ is a global solution to \eqref{EqObP}. Moreover, with \eqref{EqEqOutSide}, we have
\begin{equation}
\label{EqEqOutSide2}
F(D^2u_\infty)=1 \text{ outside } B_1.
\end{equation}

\vem

\textit{Step 2: The expansion at infinity.}

Define $v:=u_\infty-Q$ and 
\begin{equation}
\label{EqDefOfGOperator}
G(M):=F(M+A)-F(A).
\end{equation} 
By \eqref{EqEqOutSide2}, we have
$$
G(D^2v)=0 \text{ outside }B_1.
$$
Meanwhile, Proposition \ref{PropC11} gives
$
\sup_{\R^d}|D^2v|<\infty.
$
We  apply Theorem \ref{ThmLLY} and conclude
$$
v(x)=\frac12x\cdot Mx+b\cdot x+a+E(x)
$$
with $|E(x)|=O(|x|^{2-d}).$

The ordering \eqref{EqTrappedInBetween} is preserved by the convergence $u_R\to u_\infty$. As a result, we have $0\le v\le 1$ in $\R^d$. This implies
$
M=0, b=0, \text{ and }0\le a\le 1.
$
That is, 
\begin{equation}
\label{EqExpansionBeforeNontrivial}
u_\infty=Q+a+E(x),
\end{equation}
where $0\le a\le 1$ and $|E(x)|=O(|x|^{2-d}).$

The desired conclusion follows once we  show $a=0$.
\vem

\textit{Step 3: An upper barrier.}

For $r>1,$ define the auxiliary function in $\R^d\backslash\{0\}$
$$
\Sir(x):=r^{1/2}|x|^{-1/2},
$$
satisfying
\begin{equation}
\label{EqUnitValueForPsi}
\Sir=1\text{ on }\partial B_r,
\end{equation}
and 
$$
D^2\Sir(x)=\frac12r^{1/2}|x|^{-5/2}(\frac52\frac{x\otimes x}{|x|^2}-I)\hem \text{ in }\R^d\backslash\{0\}.
$$

For  the operator $G$ defined in \eqref{EqDefOfGOperator},  we compute  in $\PunS$
\begin{align*}
G(D^2\Sir)&=F(D^2\Sir+A)-F(A)\\
&=trace(DF(A)D^2\Sir)+\int_0^1[DF(A+sD^2\Sir)-DF(A)]:D^2\Sir ds.
\end{align*}
If we denote by $\omega(\cdot)$ the modulus of continuity of $DF(\cdot)$ at the matrix $A$, and use our assumption  $DF(A)=I$, then we continue as
\begin{align*}
G(D^2\Sir)&\le\Delta\Sir+\omega(|D^2\Sir|)|D^2\Sir|\\
&\le \frac12r^{1/2}|x|^{-5/2}[5/2-d+\frac52\omega(\frac54r^{1/2}|x|^{-5/2})].
\end{align*}

Fix $\overline{r}\gg 1$ such that $\frac52\omega(\frac54\overline{r}^{-2})<1/2.$ With $d\ge3$, we have
\begin{equation}
\label{EqSupersolutionPsi}
G(D^2\SiR)\le 0 \text{ outside }B_{\overline{r}}.
\end{equation} 

\vem 

\textit{Step 4: Vanishing of the constant $a$ in \eqref{EqExpansionBeforeNontrivial}.}

For $R>\overline{r}$,  take $v_R:=u_R-Q$, where $u_R$ is the solution constructed in \textit{Step 1} and $\overline{r}$ is chosen at the end of \textit{Step 3}.

With \eqref{EqTrappedInBetween} and \eqref{EqEqOutSide}, we have
$$
\begin{cases}
G(D^2v_R)=0 &\text{ in }B_R\backslash B_{\overline{r}},\\
v_R\le 1 &\text{ on }\partial B_{\overline{r}},\\
v_R=0 &\text{ on }\partial B_R.
\end{cases}
$$
Using \eqref{EqUnitValueForPsi} and \eqref{EqSupersolutionPsi}, we can apply the comparison principle and conclude
$$
u_R-Q=v_R\le \SiR \hem \text{ in }B_R\backslash B_{\overline{r}}.
$$
Sending $R\to\infty$ and using \eqref{EqExpansionBeforeNontrivial}, we have
$$
a+E=u_\infty-Q\le \overline{r}^{1/2}|x|^{-1/2} \text{ outside }B_{\overline{r}}.
$$
With $a\ge 0$ and $|E(x)|=O(|x|^{2-d})$, we conclude
$
a=0.
$

This gives the desired result, considering \eqref{EqExpansionBeforeNontrivial}.
\end{proof}


\end{document}